\documentclass[a4paper,11pt]{amsart}
\usepackage{amsmath,amssymb,epsfig,rotating}
\evensidemargin \oddsidemargin \marginparwidth 0.5in
\usepackage[all]{xy}
\usepackage{pinlabel}
\usepackage[
  pdftitle={Isotopy Classification of Engel Structures on Circle Bundles},
  pdfauthor={Mirko Klukas and Bijan Sahamie},
  citecolor=blue,
  colorlinks,
  linkcolor=blue,
  ]{hyperref}

\title[Isotopy Classification of Engel Structures]
{Isotopy Classification of Engel Structures on Circle Bundles}

\author{Mirko Klukas}
\address{Mathematisches Institut der Universit\"at zu K\"oln, 
Weyertal 86-90, 50931 K\"oln}
\urladdr{http://www.mi.uni-koeln.de/$\sim$mklukas}
\author{Bijan Sahamie}
\address{Mathematisches Institut der LMU M\"unchen, Theresienstrasse 
39, 80333 M\"unchen Germany}
\email{sahamie@math.lmu.de}
\urladdr{http://www.math.lmu.de/$\sim$sahamie}

\address{Current Address: Stanford University, 450 Sloan Hall (Building 380), Stanford, CA-94043}

\theoremstyle{plain} 
\newtheorem{theorem}{Theorem}[section]   
\newtheorem{lem}[theorem]{Lemma}         
\newtheorem{prop}[theorem]{Proposition}
\newtheorem{cor}[theorem]{Corollary}

\theoremstyle{definition}
\newtheorem{definition}[theorem]{Definition}   
\newtheorem{rem}{Remark}

\newcommand{\Z}{\mathbb{Z}}
\newcommand{\N}{\mathbb{N}}

\newcommand{\mL}{\mathcal{L}}
\newcommand{\PD}{\mbox{\rm PD}}

\newcommand{\distinv}{\mathfrak{twist}}
\newcommand{\f}{\mbox{\rm f}}
\newcommand{\smalld}{\mbox{\rm\small f}}
\newcommand{\done}{\f_{(\phi_1,\phi_2)}}
\renewcommand{\d}{\mbox{\rm d}}
\newcommand{\eng}{\mbox{\rm Eng}}

\newcommand{\mD}{\mathcal{D}}

\newcommand{\sone}{\mathbb{S}^1}
\newcommand{\lra}{\longrightarrow}
\newcommand{\co}{\colon\thinspace}
\newcommand{\lmt}{\longmapsto}

\newcommand{\pfat}{{\boldsymbol p}}
\newcommand{\Hom}{\mbox{\rm Hom}}
\newcommand{\pxi}{\mathbb{P}\xi}
\newcommand{\covn}{\mbox{\rm Cov}_n}
\newcommand{\im}{\mbox{\rm im}}

\numberwithin{equation}{section}

\DeclareMathSizes{11}{11}{8}{7}

\begin{document}

\begin{abstract} We call two Engel structures isotopic if they are
homotopic through Engel structures by a homotopy that fixes the
characteristic line field. In the present paper we define an isotopy 
invariant of Engel structures on oriented circle bundles over 
closed oriented $3$-manifolds and apply it to give an isotopy
classification of Engel structures on circle bundles with characteristic
line field tangent to the fibers. 
\end{abstract}
\maketitle
%
\fontsize{11}{14}\selectfont

\section{Introduction}
The present article deals with Engel structures on circle bundles 
over closed oriented $3$-manifolds. For an introduction to 
basic notions of Engel structures we point the reader to 
\cite[\S 2.2]{Monty}. Our focus will lie on Engel structures with 
characteristic line field tangent to the fibers of
the bundle. As mentioned in the abstract, we call two Engel 
structures {\bf isotopic} if they are homotopic through Engel 
structures via a homotopy that fixes the characteristic line field. 
This notion is motivated by a result of Golubev who proved the existence
of a version of Gray-stability for such homotopies (see~\cite{Gol} or 
\cite[Theorem~3.50]{Vogel2}).
Let $M$ be a closed oriented $3$-manifold and 
$Q\lra M$ an oriented  circle bundle. 
For an Engel structure $\mD$ 
on $Q$ with induced contact structure
$\xi$ on $M$ we denote by $\phi_\mD$ its
associated development map (see~\S \ref{sec:engelclass}). For the class of Engel
structures we are considering, Engel structures and their
development maps can be considered as equivalent objects. 
Development maps are fiberwise covering maps as
observed in \cite{KlSa} and, therefore, to understand Engel structures
on circle bundles we can equivalently try to understand fiberwise
covering maps. To this end, in 
\S\ref{sec:hordist} we define a homotopy invariant of
fiberwise covering maps we call the horizontal 
distance (see~Definition~\ref{def:hordist}). This 
distance --~as the name 
suggests~-- measures {\it how far apart} two fiberwise coverings 
are, homotopically. In \S\ref{sec:class} we proceed with 
a homotopy classification of fiberwise covering maps by applying
the horizontal distance (in arbitrary dimensions), 
where the homotopy goes through fiberwise covering maps. 
The $4$-dimensional case of fiberwise coverings
is then applied in \S \ref{sec:engelclass} to the development maps 
of Engel structures:
For Engel structures with characteristic line field tangent to the
fibers there exists a natural 
analogue of the horizontal distance we denote by 
$\distinv$ (see~\S\ref{sec:engelclass}).
Via the development maps, the invariant $\distinv$ is 
identified with the horizontal distance.  This 
allows us to apply the homotopy classification
of fiberwise covering maps to prove the 
following Theorem~\ref{thm:first}. Here, we denote by $tw(\mD)$ 
the twisting number of an Engel structure $\mD$.
\begin{theorem}\label{thm:first} Let $Q$ be an oriented circle 
bundle over a closed oriented $3$-manifold $M$. Two Engel 
structures $\mD$ and $\mD'$ with characteristic line field tangent 
to the fibers are isotopic if and only if $tw(\mD)=tw(\mD')$, 
their induced contact structures on the base agree and 
$\distinv(\mD,\mD')=0$.
\end{theorem}
The theorem provides us with a set of invariants with which we 
can decide whether or not two given Engel structures are 
isotopic. To obtain a classification, we have to 
determine which sets of invariants can be realized by Engel 
structures. To this end, denote by $\eng(Q)$ 
the isotopy classes of Engel structures on $Q$ with characteristic 
line field tangent to the fibers. 
For every contact 
structure $\xi$ on $M$ we denote by 
$\eng^n_{\xi}(Q)\subset\eng(Q)$ the isotopy classes of Engel 
structures with twisting number $n$ and induced contact structure $\xi$
on the base $M$. Similarly, we denote by 
$\eng^n_{\xi;o}(Q)\subset\eng^n_{\xi}(Q)$ the subset of 
isotopy classes of oriented Engel structures and 
$\eng^n_{\xi;no}(Q)\subset\eng^n_{\xi}(Q)$ the subset of 
non-orientable Engel structures.

\begin{theorem}\label{thm:main} Let $Q$ be an oriented circle 
bundle over a closed oriented $3$-manifold $M$.
\begin{enumerate}
\item[(i)] For a number $n$ and a contact structure $\xi$ on $M$ the set 
$\eng^n_{\xi}(Q)$ is non-empty if and only if $n\cdot e(Q)=2\cdot e(\xi)$.
\item[(ii)] The set $\eng^n_{\xi;o}(Q)$ is non-empty if and only if
$n$ is even and $n/2\cdot e(Q)=e(\xi)$.
\item[(iii)] In case $\eng^n_{\xi}(Q)$ is non-empty, there is a 
simply-transitive $H^1(M;\Z)$-action on the set $\eng^n_{\xi}(Q)$. 
Furthermore, if $\eng^n_{\xi;o}(Q)$ is non-empty, the action on
$\eng^n_{\xi}(Q)$ descends to a simply-transitive 
$2\cdot H^1(M;\Z)$-action on $\eng^n_{\xi;o}(Q)$.
\end{enumerate}
\end{theorem}
A particularly nice special case of Theorem~\ref{thm:main} are trivial 
circle bundles. In that case we obtain the following result.
\begin{cor}\label{cor:main} The set $\eng(M\times\sone)$ stays in 
one-to-one correspondence with elements in 
$\Z\times\Xi_2(M)\times H^1(M;\Z)$ where $\Xi_2(M)$ denotes the 
set of contact structures on $M$ with first 
chern class a $2$-torsion class. Moreover, the one-to-one 
correspondence establishes a bijection from the isotopy classes of 
oriented Engel structures to the subset 
$2\Z\times\Xi_0(M)\times H^1(M;\Z)
\subset\Z\times\Xi_2(M)\times H^1(M;\Z)$ 
where $\Xi_0(M)$ denotes the set of contact 
structures on $M$ with vanishing first chern class.
\end{cor}
The approach to the homotopical classification of fiberwise 
coverings chosen in this paper is very 
intuitive (see~\S\ref{sec:hordist} and \S\ref{sec:class}). As 
a small detour, we included \S\ref{sec:relations} in which we 
discuss the relation of the present results with the results from 
\cite{KlSa}. In \cite{KlSa} we provided a classification of fiberwise 
coverings up to isomorphism of coverings in a very abstract way using
spectral sequences and \v{C}ech cohomology. With our geometrically 
intuitive construction in the present article we are able to recover 
these {\it abstractly derived} results (see~\S\ref{sec:relations}) 
using a more hands-on approach.

\section{An Introductory Example}\label{sec:introexample}
In this section we briefly discuss an introductory example without
going into too many details. The purpose is to indicate the 
various ways our homotopical invariant of fiberwise coverings can
be interpreted and formulated.

\subsection{The Covering-Space Theoretic Viewpoint}\label{sec:coveringview} 
The $2$-dimensional torus $T^2$ is a circle bundle over $\sone$. 
Obviously, the map $\phi_0(\pfat,\theta)=(\pfat,n\theta)$ is a 
fiberwise $n$-fold covering. Furthermore, for an integer $\alpha\in\Z$
we can define a different fiberwise $n$-fold covering by
$\phi_\alpha(\pfat,\theta)=(\pfat,n\cdot\theta+\alpha\cdot\pfat)$. Up to
homotopy, these are all existing fiberwise $n$-fold coverings 
of the circle bundle $T^2$ by the homotopy invariance 
of $(\phi_\alpha)_*$ on the first homology $H_1(T^2;\Z)$.

Going up two dimensions, the $4$-torus $T^4$ is a circle bundle over
$T^3$. Suppose we are given fiberwise covering 
maps $\phi_i\co T^4\lra T^4$, $i=1,2$,
and a loop $\gamma\co\sone\lra T^3$. On the pullback $\gamma^*T^4$ the
map $\phi_i$ induces a fiberwise covering 
\[
  \widetilde{\phi_i}
  \co
  \gamma^*T^4\lra\gamma^*T^4,
\]
whose homotopy class is uniquely determined by the homotopy class 
of $\phi_i$. The circle bundle $\gamma^*T^4$ is trivial.
After choosing a trivialization, $\widetilde{\phi_i}$ corresponds 
to some $\phi_{\alpha^i_\gamma}$ for a suitable $\alpha^i_\gamma\in\Z$. 
It is plausible to expect that $\alpha^i_\gamma$ is invariant under 
homotopies of $\phi_i$ and that it just depends on the homology class 
of the loop $\gamma$ in $T^3$. Let us accept this for 
now (see~\S\ref{sec:twist}, \S\ref{sec:group} and the proof 
of Theorem~\ref{thm:coveringclassification}). 
Although the assignment 
$\gamma\lmt\alpha^i_\gamma$ depends on the trivialization 
of $\gamma^*T^4$, the difference $\alpha^1_\gamma-\alpha^2_\gamma$ 
is certainly independent of it. If we choose
a trivialization $\gamma^*T^4\cong T^2$, then $\widetilde{\phi_1}$ 
and $\widetilde{\phi_2}$
correspond to maps which can be written as matrices
\[
 \widetilde{\phi_1}=
 \left(
 \begin{matrix}
 1 & 0 \\
 \alpha^1_\gamma & n
 \end{matrix}
 \right)
 \;\;\mbox{\rm and }\;\;
 \widetilde{\phi_2}=
 \left(
 \begin{matrix}
 1 & 0 \\
 \alpha^2_\gamma & n
 \end{matrix}
 \right).
\]
From this we can read off the equality
\begin{equation}
\begin{array}{rcl}
 (\alpha^1_\gamma-\alpha^2_\gamma)[F]
 &=&
 (p_F)_*((\phi_1)_*-(\phi_2)_*)[\gamma]\\
&=&
 [p_F\circ\phi_1\circ\gamma]
 -
[p_F\circ\phi_2\circ\gamma],
\end{array}
 \label{eq:equation03}
\end{equation}
where $[F]$ is the fundamental class of a fiber and
$p_F\co T^4\lra\sone$ the projection onto the fiber.

\subsection{The Interpretation as Twists}\label{sec:twist} An alternative 
viewpoint can be derived 
easily from the given discussion. As mentioned above, the 
$\alpha^i_\gamma$'s just depend on the homotopy type of
$\phi_i$ and the homology class of the $\gamma$'s. Hence, the assignment
$\gamma\lmt\alpha^1_\gamma-\alpha^2_\gamma$ can be reduced to a map 
\begin{equation}
  \d(\phi_1,\phi_2)\co 
  H_1(T^3;\Z)\lra\Z,\,
  [\gamma]\lmt\alpha^1_\gamma-\alpha^2_\gamma,
 \label{eq:equation04}
\end{equation}
which carries the relevant homotopical information. So,
\begin{equation}
  \phi_\alpha(\pfat,\theta)
  =
  (\pfat,n\theta+\left<\alpha,\pfat\right>),
  \label{eq:equation05}
\end{equation}
where $\alpha=(\alpha_1,\alpha_2,\alpha_3)\in\Z^3$, all 
represent different homotopy classes of
fiberwise $n$-fold covering maps. Comparing with the discussion 
from \S\ref{sec:coveringview}, we see that the $\phi_\alpha$ have
to be distinct, pairwise. Namely, for a standard basis $[\gamma_i]$, 
$i=1,2,3$, of $H_1(T^3;\Z)$ with $\phi_1=\phi_\alpha$ and 
$\phi_2=\phi_0$ we compute
\[
 \d(\phi_\alpha,\phi_0)[F]
 =(\alpha^1_{\gamma_i}-\alpha^2_{\gamma_i})[F]
 =\alpha_i[F],
\]
where the first equality is given by definition of 
$\d(\phi_\alpha,\phi_0)$ and the 
second equality is given by \eqref{eq:equation03}.
So, $\d(\phi_\alpha,\phi_0)$ is the morphism on $H_1(T^3;\Z)$ which sends 
$[\gamma_i]$ to $\alpha_i$ for $i=1,2,3$. Moreover, 
the right hand side of the Equation~\eqref{eq:equation03} we used here 
determines the number of times $\phi_1$ moves around the fiber relative 
to $\phi_2$ as we move along $\gamma$. We see by 
\eqref{eq:equation04} and Equation~\eqref{eq:equation03}
that this is precisely what is measured by $\d(\phi_\alpha,\phi_0)$.

\subsection{The Group Action}\label{sec:group} Another way 
of formulating the definition of $\phi_\alpha$ is
\[
 \phi_\alpha(\pfat,\theta)
 =(\pfat,n\theta+\left<\alpha,\pfat\right>)
 =\left<\alpha,\pfat\right>\cdot(\pfat,n\theta)
 =\left<\alpha,\pfat\right>\cdot\phi_0(\pfat,\theta),
\]
where the second equality uses the $\sone$-action on the fibers 
of $T^4$. If we define a map $\f\co T^3\lra \sone$ by 
$\f(\pfat)=\left<\alpha,\pfat\right>$, then 
$\phi_\alpha(\pfat,\theta)=\f(\pfat)\cdot\phi_0(\pfat,\theta)$. In fact, every map 
$\phi_\alpha$ can be obtained from $\phi_0$ by using a 
suitable map $T^3\lra\sone$ and the group action as above. It 
is not hard to see that a homotopy of $\phi_\alpha$ through 
fiberwise covering maps is equivalent to a homotopy of the map 
$\f$. The homotopy type of $\f$ is captured by the class 
$\f^*[\sone]\in H^1(T^3;\Z)$ where $[\sone]$ is a fundamental 
class of $\sone$. By the universal coefficient theorem, 
$H^1(T^3;\Z)$ is isomorphic to $\Hom(H_1(T^3;\Z);\Z)$. Therefore, 
the class $\f^*[\sone]$ is uniquely determined by evaluation 
on a basis of $H_1(T^3;\Z)$. Using the standard basis 
$[\gamma_i]$, $i=1,2,3$, of the first homology of $T^3$, we obtain
\[
 (\f^*[\sone])[\gamma_i]=\alpha_i
\]
for $i=1,2,3$. More generally, given two fiberwise coverings 
$\phi_i$, $i=1,2$, and let $\f\co T^3\lra\sone$ be given such 
that $\phi_2=\f\cdot\phi_1$, then for every loop $\gamma$ we 
have
\[
 \alpha^1_\gamma-\alpha^2_\gamma=(\f^*[\sone])[\gamma].
\]
So, in particular, $\d(\phi_1,\phi_2)[\gamma]=(\f^*[\sone])[\gamma]$.

\subsection{Relation to Engel Structures on the Four-Dimensional Torus}
\label{sec:exaengel} It 
is easily possible to define Engel structures on the four-dimensional 
torus. Let $\xi$ be the contact structure on the base space 
which is given as the kernel of the $1$-form $\sin(2\pi z)dx + \cos(2\pi z)dy$. 
The contact planes are spanned by the vector field $\partial_z$ and
\[
  V_\pfat
 =
 \cos(2\pi z)\partial_x+\sin(2\pi z)\partial_y
\]
where $\pfat$ is a point in $T^3$. On $T^4$, with coordinates $(\pfat,\theta)$, we 
define an Engel structure $\mD^n_\alpha(\xi)$ as the $2$-planes spanned by 
$\partial_\theta$ and
\[
 \cos\big(\pi\big(n\theta + \left<\alpha,\pfat\right>\big)
 \big)
 \thinspace
 \partial_z
 +
 \sin\big(\pi\big(n\theta + \left<\alpha,\pfat\right>\big)
 \big)
 \thinspace
 V_\pfat. 
\] 
There is a natural map $\phi\co T^4\lra T^4$ called the development 
map. The development map of this Engel structure is the fiberwise 
covering map $\phi_\alpha$. We should see that the data of the Engel 
structure translate into the data that determine the homotopy type of
$\phi_\alpha$: Pick the standard basis $[\gamma_i]$, $i=1,2,3$, of
$H_1(M;\Z)$ 
represented by the obvious loops $\gamma_i$, then 
$\left.\mD^n_\alpha(\xi)\right|_{\gamma_i}$ determines a family 
of $1$-dimensional subspaces in $\left.\xi\right|_{\gamma_i}$. What we
see is that by the summand $\left<\alpha,\,\cdot\,\right>$ the subspaces
of $\left.\xi\right|_{\gamma_i}$ 
associated to $\left.\mD^n_\alpha(\xi)\right|_{\gamma_i}$ make
$\alpha_i$-full turns inside $\left.\xi\right|_{\gamma_i}$ relative to
the family of subspaces given by $\left.\mD^n_0(\xi)\right|_{\gamma_i}$.

\section{Homotopical Classification of Fiberwise Coverings}
For this section suppose we are given a closed oriented manifold 
$M$ and two circle bundles $Q$ and $P$ over $M$. Recall that a 
map $\phi\co Q\lra P$ is called a 
fiberwise $n$-fold covering map if its restriction to every fiber 
$Q_p$, $p\in M$, equals the standard $n$-fold covering map 
$\varphi_n\co\sone\lra\sone,\,\theta\lmt n\cdot\theta$. To be more 
precise, recall that there are simply-transitive $\sone$-actions on 
the fibers of $Q$ and $P$. Fixing an element in the fiber $Q_p$ and $P_p$, 
the group 
actions provide us with an identification of $Q_p$ and $P_p$ with 
$\sone$. Under these identifications, the map 
$\left.\phi\right|_{Q_p}$ corresponds to $\varphi_n$. 
\begin{rem} A priori, one 
could define fiberwise $n$-fold coverings by requiring the restriction 
$\left.\phi\right|_{Q_p}$ to be an $n$-fold covering but not specifically 
$\varphi_n$. However, these two notions are equivalent up to homotopy. 
In fact, a fiberwise covering of the {\it more general form} can be homotoped 
into a fiberwise covering of the {\it restricted type} we are considering in
this article.
\end{rem}
\subsection{The Horizontal Distance}\label{sec:hordist} Let $G$ be 
an abelian group,
$n$ a natural number and denote by $K(G,n)$ the associated 
Eilenberg-MacLane space. Recall that for every CW-complex $X$ 
there is a natural bijection
\[
 [X;K(G,n)]\lra H^n(X;\Z), \, [f]\lmt f^*(\alpha_0),
\]
where $\alpha_0\in H^n(K(G,n);G)$ is a fixed fundamental class 
of $K(G,n)$ (see~\cite[Theorem~4.57]{Hatchy}). We will apply 
this classification to fiberwise 
coverings: Namely, given two $n$-fold fiberwise covering maps 
$\phi_i\co Q\lra P$, $i=1,2$, for every $p\in M$ we define 
$\done(p)\in \sone$ as the element defined via the equation
\[
  \done(p)\cdot \phi_1(q)=\phi_2(q),
\]
where $q$ is an arbitrary point in $Q_p$. This 
definition does not depend on the point $q$. If $q'$ is another 
point in the fiber over $p$, then there exists an element 
$\theta\in\sone$ such that $\theta\cdot q=q'$. By the chain of 
equalities
\[
  \begin{array}{rcl}
  \phi_1(q')
  &=&
 \phi_1(\theta\cdot q)
 =(n\cdot\theta)\cdot f(q)
 =(n\cdot\theta+ \done(p))\cdot\phi_2(q')\\ 
 &=&
 (\done(p)+n\cdot\theta)\cdot\phi_2(q)
 =\done(p)\cdot \phi_2(\theta\cdot q)\\
 &=&
 \done(p)\cdot\phi_2(q')
\end{array}
\]
well-definedness follows. Here, we applied the assumption that 
$\phi_1$ and $\phi_2$ are both fiberwise covering maps. Thus, 
we are provided with a map
\[
  \done\co M\lra\sone.
\]
We will prove that the homotopy type of this map 
{\it measures the homotopical distance} between $\phi_1$ and 
$\phi_2$. To this end, we need a homotopical classification
of maps into $\sone$. Note that $\sone$ equals $K(\Z,1)$ 
and, therefore, we have a bijection
\[
  [M;\sone]\lra H^1(M;\Z),\, [f]\lmt f^*([\sone]), 
\]
where $[\sone]$ is the fundamental class of $[\sone]$ that 
represents the natural orientation on the $\sone$-fibers.
\begin{definition}\label{def:hordist} Given two fiberwise $n$-fold 
coverings $\phi_i\co Q\lra P$, $i=1,2$, we denote by  $\d(\phi_1,\phi_2)$ 
the class $\done^*([\sone])\in H^1(M;\Z)$ and call it
{\bf the horizontal distance} between $\phi_1$ and $\phi_2$.
\end{definition}
The example of the $4$-torus discussed in 
\S\ref{sec:introexample} shows that different fiberwise 
coverings can be constructed by adding twists along loops in 
the base space (see~\S\ref{sec:twist}). More precisely, 
we have seen in the specific case of the $4$-torus that the
horizontal distance between two fiberwise coverings
$\phi_1$ and $\phi_2$ measures some sort of twisting 
of $\phi_1$ relative to $\phi_2$ 
along loops in the base space (see~\S\ref{sec:coveringview} 
and \S\ref{sec:twist}). What is done for the $4$-torus can
be done in the general setting as well. Moreover, observe
that
$\d(\phi_1,\phi_2)\frown[\gamma]=\d(\phi_1,\phi_2)([\gamma])$ 
where on the right we interpret the class $\d(\phi_1,\phi_2)$ 
as an element of $\mbox{\rm Hom}(H_1(M;\Z),\Z)\cong H^1(M;\Z)$ 
by the universal coefficient theorem. However, we have 
\[
  \d(\phi_1,\phi_2)[\gamma]=\done^*([\sone])[\gamma]
 =[\done\circ\gamma]
\]
and the homotopy class of $\done\circ\gamma\co\sone\lra\sone$ 
precisely measures the relative twisting discussed in \S\ref{sec:twist}.

\subsection{Homotopy Classification}\label{sec:class} Let 
$Q$ and $P$ be two circle bundles
over a closed oriented manifold. Denote by $\mbox{\rm Cov}(Q,P)$
the homotopy classes of fiberwise covering maps where homotopies
move through fiberwise coverings and for every $n\in\N$ denote by
$\mbox{\rm Cov}_n(Q,P)$ the subset consisting of homotopy classes of
$n$-fold coverings. Our goal is to prove the 
following theorem.
\begin{theorem}\label{thm:coveringclassification} For a 
number $n\in\N$, the set $\covn(Q,P)$ is
non-empty if and only if $n\cdot e(Q)=e(P)$. 
If it is non-empty,  there 
is a simply-transitive group action of $H^1(M;\Z)$
on $\covn(Q,P)$. Hence, there is a 
one-to-one correspondence between $\covn(Q,P)$ 
and $H^1(M;\Z)$.\end{theorem}
In order to prove this result, we will need the following two 
preparatory lemmas.

\begin{lem}\label{lem:homotopy} Suppose we are given two fiberwise 
covering maps $\phi_i\co Q\lra P$, $i=1,2$. These maps are homotopic 
through fiberwise covering maps if and only if they have the same 
number of sheets and their horizontal distance vanishes.
\end{lem}
\begin{proof} If the maps are homotopic as fiberwise covering 
maps, then the map $\f_{(\phi_1,\phi_2)}$ is homotopic to a 
constant map. Hence, their horizontal distance vanishes. 

Conversely, 
suppose that $\phi_i$, $i=1,2$, have the same number of sheets and
have vanishing horizontal distance. We define a homotopy
\[
  F\co Q\lra [0,1]\lra Q,
  \,(q,t)\lmt h(\pi(q),t)\cdot\phi_1(q),
\]
where $h\co M\times [0,1]\lra \sone$ is a homotopy from 
$\f_{(\phi_1,\phi_2)}$ to the constant map $c_1$ which 
maps every point to $1$.
\end{proof}
\begin{lem}\label{lem:additivity} Given fiberwise covering 
maps $\phi_i\co Q\lra P$, $i=1,2,3$, of a circle bundle 
$P\lra M$ over a closed oriented manifold $M$ we have 
$\d(\phi_1,\phi_2)+\d(\phi_2,\phi_3)=\d(\phi_1,\phi_3)$.
\end{lem}
\begin{proof} Given a point $q\in Q$ we have 
$\f_{(\phi_1,\phi_3)}(q)\cdot \phi_1(q)=\phi_3(q)$.
Furthermore, $\f_{(\phi_1,\phi_3)}(q)$ is uniquely 
determined by this equality. However, 
\begin{eqnarray*}
 (\f_{(\phi_1,\phi_3)}(q)
 +\f_{(\phi_1,\phi_2)}(q))\cdot\phi_1(q)
 &=&
 \f_{(\phi_2,\phi_3)}(q)
 \cdot(\f_{(\phi_1,\phi_2)}(q)\cdot\phi_1(q))
 =
 \f_{(\phi_2,\phi_3)}(q)\cdot
 \phi_2(q)\\
 &=&\phi_3(q)\\
 &=&\f_{(\phi_1,\phi_3)}(q)\cdot\phi_1(q).
\end{eqnarray*}
Since the action on the fibers is simply-transitive, 
the sum
$\f_{(\phi_2,\phi_3)}(q)+\f_{(\phi_1,\phi_2)}(q)$
equals
 $\f_{(\phi_1,\phi_3)}(q)$. 
Hence, 
\[
  \d(\phi_1,\phi_3)=\f_{(\phi_1,\phi_3)}^*[\sone]
 =
 \f_{(\phi_2,\phi_3)}^*[\sone]
 +\f_{(\phi_1,\phi_2)}^*[\sone]
 =
 \d(\phi_1,\phi_2)+\d(\phi_2,\phi_3),
\]
which finishes the proof.
\end{proof}

\begin{proof}[Proof of 
Theorem~\ref{thm:coveringclassification}] Given a fiberwise 
covering $\phi_1\co Q\lra P$ and a class 
$\alpha\in H^1(M;\Z)$, denote by $\f\co M\lra\sone$ a map 
with $\f^*[\sone]=\alpha$. Then, define
\begin{equation}
  \f\cdot\phi_1\co Q\lra P,\; q\lmt \f(\pi(q))\cdot \phi_1(q),
  \label{eq:action}
\end{equation}
where $\pi\co Q\lra M$ is the canonical projection map. We define 
$\alpha\cdot [\phi_1]:=[\f\cdot\phi_1]$. The homotopy class 
of $\f\cdot\phi_1$ does not depend on the specific choice 
of $\f$ which can be derived directly using 
Lemma~\ref{lem:additivity}. So, we are provided with a map
\[
  \kappa\co H^1(M;\Z)\lra \mbox{\rm Cov}_n(Q,P),
  \;\alpha\lmt \alpha\cdot[\phi_1].
\]
By construction, $\d(\phi_1,\kappa(\alpha))=\alpha$. So, the 
last two lemmas show that this is a bijection: Namely, 
given $\alpha$, $\alpha'$ such that 
$\kappa(\alpha)=\kappa(\alpha')$ then 
\[
  \alpha-\alpha'
  =\d(\phi_1,\kappa(\alpha))
  +\d(\kappa(\alpha'),\phi_1)
  =\d(\phi_1,\phi_1)=0
\]
by Lemma~\ref{lem:additivity}. Hence, $\alpha=\alpha'$ which 
shows injectivity.

To prove surjectivity, pick a fiberwise $n$-fold covering 
$\phi_2\co Q\lra P$ and denote by $\alpha$ the class 
$\d(\phi_1,\phi_2)$. Now we have
\[
  \d(\phi_2,\kappa(\alpha))
  =\d(\phi_2,\phi_1)
  +\d(\phi_1,\kappa(\alpha))=0
\]
by Lemma~\ref{lem:additivity}. Finally, Lemma~\ref{lem:homotopy} 
shows that the fiberwise coverings $\phi_2$ and 
$\kappa(\alpha)$ are homotopic through fiberwise covering 
maps.\vspace{0.3cm}

To prove the existence statement in the theorem, assume 
that $n\cdot e(Q)=e(P)$. Denote by $\Sigma\subset M$ a 
submanifold which represents the homology class 
$\PD[e(Q)]\in H_{n-2}(M;\Z)$ and denote by $\nu\Sigma$ 
its normal bundle. Since $U(1)$ is a retract of 
$\mbox{\rm Diff}(\sone)$, without loss of generality we 
may assume that $Q$ determines a two-dimensional
real vector bundle 
$E_Q\lra M$ for which $Q$ is the unit-sphere bundle after
choosing a suitable metric $g$. The 
homology class of $\Sigma$ is the homology class of the 
zero locus of a section $s\co M\lra E_Q$ which intersects 
the zero section of $E_Q$ transversely. So, we may assume 
that $s^{-1}(0)=\Sigma$. Consequently, by normalizing $s$, 
we obtain a section of 
$\left. Q\right|_{M\backslash\nu\Sigma}$. Therefore, we 
can write $Q$ as
\begin{equation}
 Q
   =
 (M\backslash\nu\Sigma)\times\sone
 \cup_\psi
(\nu\Sigma\times\sone)
\label{eq:equation01}
\end{equation}
where 
$
 \psi
 =
 (\psi_1,\psi_2)\co 
 \partial\nu\Sigma\times\sone 
 \lra
 \partial(M\backslash\nu\Sigma)\times\sone
$ 
denotes a suitable gluing map. We define a circle bundle
$P_\eta$ by
\begin{equation}
 P_\eta
   =
 (M\backslash\nu\Sigma)\times\sone
 \cup_{\eta}
(\nu\Sigma\times\sone)
\label{eq:equation2}
\end{equation}
where $\eta$ is the gluing map $(\psi_1,\psi_2^n)$. The Euler 
class of $P_\eta$ is $n\cdot e(Q)$ which equals the 
Euler class of $P$. So, $P_\eta$ is isomorphic to $P$ and 
we may use $P_\eta$ as a model for $P$. Now, there is an obvious 
fiberwise $n$-fold covering map from $Q$ to $P$, namely the 
map
\[
  \phi\co Q\lra P,\, [(p,\theta)]\lmt [(p,n\cdot\theta)].
\]
Here, we used the description \eqref{eq:equation01} of $Q$ and the
description \eqref{eq:equation2} of $P$.

Conversely, assuming that there is a fiberwise $n$-fold 
covering $\phi\co Q\lra P$, we have to prove that 
$n\cdot e(Q)=e(P)$. We either use a similar reasoning as 
above, or we apply \cite[Lemma~3.1]{KlSa}.
\end{proof}

\subsection{Relations to Previous Results}\label{sec:relations} The 
horizontal distance also detects if two fiberwise covering maps 
$\phi_i\co Q\lra P$, $i=1,2$, are isomorphic as fiberwise covering 
maps.
\begin{prop}\label{prop:characiso} Two fiberwise coverings 
$\phi_i\co Q\lra P$, $i=1,2$, are isomorphic as fiberwise coverings 
if and only if they have the same number $n$ of sheets and 
$\d(\phi_1,\phi_2)\;\mbox{\rm mod } n=0$.
\end{prop}
\begin{proof} Suppose that $\phi_i$, $i=1,2$, are isomorphic as
fiberwise $n$-fold covering maps and denote by $\psi\co Q\lra Q$ 
an isomorphism. It is easy to see that $\psi$ also commutes with
the bundle projection $Q\lra M$. Hence, $\psi$
is also an automorphism of the circle bundle $Q\lra M$. Every such
bundle isomorphism can be characterized by a map 
$\psi_2\co M\lra\sone$ by the equation
\[
  \psi(q)=\psi_2(\pi(q))\cdot q.
\]
Note that $\psi_2$ is well-defined
since $\psi$ maps fibers to fibers. We have to compute 
$\d(\phi_1,\phi_2)=\d(\phi_1,\phi_1\circ\psi)$. For a point 
$q\in Q$ consider
\[
  \f_{(\phi_1,\phi_1\circ \psi)}(q)\cdot \phi_1(q)=
  \phi_1(\psi(q))=\phi_1(\psi_2(q)\cdot q)=
  (n\cdot \psi_2(q))\cdot \phi_1(q),
\]
which implies that 
$\f_{(\phi_1,\phi_1\circ \psi)}(q)=n\cdot \psi_2(q)$. Consequently, given 
a loop $\gamma$ in $Q$ then 
$
 (\f_{(\phi_1,\phi_1\circ \psi)})_*([\gamma])
 =n\cdot(\psi_2)_*[\gamma]
$. 
Considering the diagram
\[
  \xymatrix@R=1cm@C=1cm{
  H^1(\sone;\Z)
  \ar[r]^{\hspace{-0.7cm}\cong}
  \ar[d]_{\smalld_{(\phi_1,\circ \phi_1\circ \psi_2)}^*} & 
  \mbox{\rm Hom}(H_1(\sone;\Z),\Z)
  \ar[d]^{((\smalld_{(\phi_1,\phi_1\circ \psi)})_*)^*}
  \ar@{=}[r]&
  \mbox{\rm Hom}(H_1(\sone;\Z),\Z)
  \ar[d]_{(n\cdot(\psi_2)_*)^*} &
  \ar[l]_{\hspace{0.7cm}\cong} 
  H^1(\sone;\Z)\ar[d]^{n\cdot\psi_2^*} \\
  H^1(Y;\Z)\ar[r]^{\hspace{-0.7cm}\cong} & 
  \mbox{\rm Hom}(H_1(Y;\Z),\Z)
  \ar@{=}[r] & 
  \mbox{\rm Hom}(H_1(Y;\Z),\Z) & 
  \ar[l]_{\hspace{0.7cm}\cong} 
  H^1(Y;\Z) 
}
\]
our discussion shows that the square in the middle commutes.
By the universal coefficient theorem the left and the right 
square of the diagram commute. Thus, 
\[
  \d(\phi_1,\phi_1\circ\psi_1)
  =n\cdot\psi_2^*[\sone]
\]
which makes it vanish modulo-$n$.

Conversely, given two fiberwise $n$-fold coverings 
$\phi_i\co Q\lra P$, $i=1,2$, such that 
$\d(\phi_1,\phi_2)$ vanishes modulo $n$ this means that 
$\d(\phi_1,\phi_2)=n\cdot\alpha$ for some suitable 
class $\alpha\in H^1(M;\Z)$. Denote by $\psi_2\co M\lra\sone$ 
a continuous map for which $\psi_2^*[\sone]=\alpha$. We
construct an associated automorphism $\psi$ of the bundle 
$Q\lra M$ by sending an element $q\in Q$ to 
$\psi(q)=\psi_2(\pi(q))\cdot q$. By the considerations 
from above, we get $\d(\phi_1,\phi_1\circ\psi)=n\cdot\alpha$. 
Consequently, we have
\[
  \d(\phi_1\circ\psi,\phi_2)
  =\d(\phi_1\circ\psi,\phi_1)
  +\d(\phi_1,\phi_2)
  =-n\cdot\alpha+n\cdot\alpha
  =0,
\]
where the first and second equality holds by 
Lemma~\ref{lem:additivity}.
Theorem~\ref{thm:coveringclassification} implies
that $\phi_1\circ\psi$ and $\phi_2$ are homotopic. 
So, after possibly defining $\psi$ (resp.~$\psi_2$) 
differently,  $\phi_1\circ\psi$ equals $\phi_2$ which finishes 
the proof.
\end{proof}
\begin{rem} The last step in the proof of 
Proposition~\ref{prop:characiso} used the fact that a 
homotopy of $\phi_1\circ\psi$ through fiberwise covering 
maps amounts to a homotopy of $\psi_2$. We leave the proof 
of this fact to the interested reader.
\end{rem}
In \cite{KlSa} we already identified a 
cohomology class which provides a classification of 
fiberwise covering maps up to isomorphism of coverings. 
However, the definition of the class provided there was on 
a very abstract level using spectral sequences and \v{C}ech 
cohomology. The following result shows that our geometrically 
defined class $\d(\phi_1,\phi_2)$ reduced modulo $n$ equals 
the class defined in \cite{KlSa} 
(see~\cite[Theorem~3.3]{KlSa} and cf.~\cite[Corollary~4.1]{KlSa}).
\begin{prop}\label{prop:relateold} Given two fiberwise 
$n$-fold coverings $\phi_i\co Q\lra P$, $i=1,2$, 
and let $\alpha_{\phi_i}\in H^1(M;\Z_n)$, $i=1,2$, 
be their associated classes defined via \v{C}ech cohomology 
using the approach in \cite{KlSa}. The classes
$\alpha_{\phi_1}-\alpha_{\phi_2}$ and $\d(\phi_1,\phi_2)\;\mbox{\rm mod } n$
are equal. 
\end{prop}
\begin{proof}[Sketch of the Proof] In \cite[Theorem~1.1]{KlSa} we gave a 
characterization of fiberwise $n$-fold coverings via the 
following exact sequence.
\[
  \xymatrix{
  0\ar[r] &
  H^1(M;\Z_n)\ar[r]^{\pi^*} &
  H^1(P;\Z_n)\ar[r]^{\iota^*} &
  H^1(\sone;\Z_n)\ar[r]^{d} &
  H^2(M;\Z_n).
  }
\]
We have shown that fiberwise coverings correspond to classes 
$\beta\in H^1(P;\Z_n)$ for which 
$\iota_*(\beta)=[\varphi_n]$ where $\varphi_n$ is the unique (up to homotopy)
connected $n$-fold covering of $\sone$. Thus, there is a simply 
transitive group action of $H^1(M;\Z_n)$ on the set of $n$-fold 
fiberwise coverings given via
\[
  \begin{array}{rclcl}
  H^1(M;\Z_n)&\times&(\iota_*)^{-1}([\varphi_n])
  &\lra&
  (\iota_*)^{-1}([\varphi_n])\\
  (\alpha&,&\beta)
  &\lmt & 
  \beta+\pi^*(\alpha).
  \end{array}
\]
Choose an open covering $\mathcal{U}$ of $M$ and let us 
denote by $N(\mathcal{U})$ the nerve of the covering. Without 
loss of generality the covering $\mathcal{U}$ can be chosen 
fine enough such that $\pi^*(\alpha)$ is the image of a 
suitable class, $\alpha_\mathcal{U}$ say, in 
\[
  H^1(N(\mathcal{U});\Z)
  \lra
  \varinjlim H^1(N(\mathcal{V}),\Z)
  =
  \varinjlim\mbox{\rm Hom}(H_1(N(\mathcal{V});\Z),\Z),
\]
where the last equality is provided by the universal 
coefficient theorem and the naturality of the universal 
coefficient theorem exact sequence. Therefore, the class 
$\alpha_\mathcal{U}$ can be seen as a morphism on the first 
homology of the nerve and, as such, it is characterized by 
its image on a basis of $H_1(N(\mathcal{U});\Z)$. 
With a little
effort one can see that $\alpha_\mathcal{U}$ on $H_1(N(\mathcal{U});\Z)$ 
has an analogous description as the horizontal distance 
on $H^1(M;\Z)$. Furthermore, it is not hard to see that this 
description is preserved by the direct limit, which 
shows that $\alpha_\mathcal{U}$ converges to a class whose
definition coincides with the definition of the horizontal distance.
\end{proof}

\section{Engel Structures on Circle Bundles}\label{sec:engelclass}
Given a circle bundle $Q\lra M$ over a closed oriented $3$-manifold $M$
with Engel structure $\mD$ whose characteristic line field is tangent to the
fibers, we associate to it the so-called development map
\[
  \phi_{\mD}\co (Q,\mD)\lra (\pxi,\mD\xi)
\]
where $(\pxi,\mD\xi)$ is the prolongation of the contact structure 
$\xi$ induced on the base space $M$ (see~\cite[\S 4]{KlSa}). We 
have seen in \cite{KlSa} that $\phi_\mD$ is a fiberwise covering 
map where the twisting 
number $tw(\mD)$ equals the number of sheets of this covering. An 
isotopy of Engel structures, i.e.~a homotopy through Engel structures 
whose characteristic line field is tangent to the fibers, leaves 
fixed the induced contact structure on the base. Hence, an isotopy 
of the Engel structure $\mD$ is equivalent to a homotopy of 
$\phi_\mD$ through fiberwise covering maps (see~Theorem~\ref{thm:first}). 
As a result, \S\ref{sec:class} provides us with a method to 
classify such Engel structures up to isotopy. 

Given another Engel structure $\mD'$ then to every loop $\gamma$ 
in $M$ we can associate an integer in the following way: Lift the 
loop $\gamma$ to a loop $\widetilde{\gamma}$ in $Q$. The Engel 
structures $\mD$ and $\mD'$ induce $\sone$-families of $1$-dimensional
subspaces of the trivial bundle 
$\left.\xi\right|_{\pi(\widetilde{\gamma})}$. These
paths of subspaces can be interpreted as sections of 
$\left.\pxi\right|_{\gamma}$. If 
we choose the path coming from $\mD$ as a reference, it provides us
with an identification $\left.\pxi\right|_{\gamma}\cong\gamma\times\sone$.
With this identification,  
to the path coming from $\mD'$ we can assign an integer by measuring 
the number times the path associated to $\mD'$ moves through the
$\sone$-fiber of $\left.\pxi\right|_{\gamma}$. This number shall be denoted by 
$\distinv(\mD,\mD')(\gamma)$.  In fact, we see that 
\begin{equation}
  \distinv(\mD,\mD')(\gamma)=\d(\phi_\mD,\phi_{\mD'})_*[\gamma].
  \label{eq:compare}
\end{equation}
Thus, the given assignment descends to a map
\[
  \distinv(\mD,\mD')\co H_1(M;\Z)\lra \Z.
\]
The invariant $\distinv$ is the invariant coming from the 
intuition that the Engel structure on a circle bundle with 
characteristic line field tangent to the fibers is basically given 
by the contact structure on a cross section, its twisting 
along fibers (vertical twisting), and its twisting along non-trivial 
loops in the base space (horizontal twisting, cf.~\S\ref{sec:exaengel}). The 
Equality~\eqref{eq:compare} indicates that this is an isotopy invariant 
and that it should be strong enough to give a 
classification.

\begin{proof}[Proof of Theorem~\ref{thm:first}] Given two Engel
structures $\mD_0$ and $\mD_1$ which are isotopic and denote by
$(\mD_t)_{t\in[0,1]}$ a corresponding homotopy through Engel 
structures. There exists an isotopy $\psi_t\co Q\lra Q$
such that $(\psi_t)_*\mD_0=\mD_t$ (see~\cite{Gol} or 
\cite[Theorem~3.50]{Vogel2}). Since the homotopy fixes the
characteristic line field, the isotopy $\psi_t$ is an automorphism
of the circle bundle $Q\lra M$. Denote by
$\mathcal{E}_t$ the associated even contact structure of $\mD_t$.
Since $\mathcal{E}_t=[\mD_t,\mD_t]$ by definition, we have
\[
  \mathcal{E}_t
  =
  [\mD_t,\mD_t]
  =
  [(\psi_t)_*(\mD_0),(\psi_t)_*(\mD_0)]
  =
  (\psi_t)_*[\mD_0,\mD_0]
  =
  (\psi_t)_*\mathcal{E}_0.
\]
Thus, 
\[
  \xi_t
  =
  \pi_*(\psi_t)_*(\mathcal{E}_0)
  =
  (\pi\circ\psi_t)_*\mathcal{E}_0
  =
  \pi_*\mathcal{E}_0
  =
  \xi_0.
\]
So, the homotopy of Engel structures apparently fixes the contact structure on the
base space. Therefore, the triangle
\[
 \xymatrix@R=1cm@C=1cm{
 Q\ar[rr]^{\psi_t}\ar[rd]_{\phi_{\mD_0}} & & Q\ar[ld]^{\phi_{\mD_t}}\\
 & \pxi_0 &
 }
\]
commutes for every $t\in[0,1]$, which shows that the development 
maps $\phi_{\mD_i}$, 
$i=0,1$, are homotopic through fiberwise covering maps. Thus,
$\distinv(\mD_0,\mD_1)=\d(\phi_{\mD_0},\phi_{\mD_1})=0$.

Conversely, if $\distinv(\mD_0,\mD_1)=\d(\phi_{\mD_0},\phi_{\mD_1})=0$, the 
two development maps $\phi_{\mD_0}$ and  $\phi_{\mD_1}$ are homotopic
through fiberwise covering maps. Denote by $(\phi_t)_{t\in[0,1]}$ the associated
homotopy and denote by
$\mD\xi_0$ the prolonged Engel structure on $\pxi_0$. For every $t\in[0,1]$
there exists a unique distribution of $2$-planes $\mD_t$ on $Q$ such that
$(\phi_t)_*\mD_t=\mD\xi_0$. The family of $2$-planes 
$(\mD_t)_{t\in[0,1]}$ defines a homotopy
of Engel structures from $\mD_0$ to $\mD_1$.
\end{proof}
An Engel structure 
$\mD$ on $Q$ with characteristic line field tangent to the fibers 
induces a contact structure $\xi$ on the base space $M$. To the 
contact structure we can associate its prolongation $(\pxi,\mD\xi)$. 
As noted at the beginning of this section, to the Engel structure 
$\mD$ we can associate the development map $\phi_\mD\co Q\lra\pxi$ 
which is a fiberwise $tw(\mD)$-fold covering. An isotopy of the 
Engel structure $\mD$, i.e.~a homotopy through Engel structures 
which fixes the characteristic line field, is equivalent to a 
homotopy of the development map through fiberwise covering 
maps (see~Theorem~\ref{thm:first}). 
Hence, the existence and classification of Engel structures up to 
isotopy --~where we consider Engel structures with characteristic line 
field tangent to the fibers~-- is equivalent to the existence and 
classification of fiberwise covering maps $Q\lra P$ up to homotopy 
through fiberwise coverings. In case of oriented Engel structures 
this approach has a refinement as follows: An orientable Engel 
structure induces a trivialization of the tangent bundle $TQ$ via 
vector fields $\partial_\mL$, $\partial_\mD$, $\partial_\mathcal{E}$ 
and $\partial_{TQ}$ where the first is tangent to the characteristic 
line field and the second lies inside the Engel structure such that
$\partial_\mL$ and $\partial_\mD$ define a trivialization of the 
Engel structure (see~\cite[Theorem~3.37]{Vogel2}). 
\begin{definition} The map $\phi_1\co Q\lra \xi_1$ which sends
a point $q\in Q$ to the 
element $\phi_1(q)\in (\xi_1)_q$ associated to the subspace 
$T_q\pi(\left.\partial_\mD\right|_q)$ is called 
{\bf oriented development map} of the Engel structure $\mD$.
\end{definition}
There is a 
canonical projection map $\pi_\xi\co\xi_1\lra\pxi$ and the triangle
\[
  \xymatrix@R=1cm@C=1cm{
  Q\ar[r]^{\phi_\mD}\ar[d]_{\phi_1} & \pxi \\
  \xi_1 \ar[ru]_{\pi_\xi}
}
\]
obviously commutes. This implies that the twisting 
number $n$ is even and that the equation 
$n/2\cdot e(Q)=e(\xi_1)=e(\xi)$ holds 
(see~Theorem~\ref{thm:coveringclassification}). 

Conversely, if the twisting 
number is even and the development map factorizes through the unit 
sphere bundle $\xi_1$, then $\mD$ is orientable. Since this 
factorization is canonical upon a choice of Riemannian metric on $Q$, 
an isotopy of oriented Engel structures 
corresponds to a homotopy of $\phi_1$ through fiberwise covering 
maps.
\begin{proof}[Proof of Theorem~\ref{thm:main}] 
An Engel structure $\mD^n(\xi)$ on $Q$ with twisting 
number $n$ and induced contact structure $\xi$ exists if 
and only if $n\cdot e(Q)=e(\pxi)$. Since $e(\pxi)$ equals 
$2\cdot e(\xi)$ we see that $\mD^n(\xi)$ exists if and only if 
$n\cdot e(Q)=2\cdot e(\xi)$ which proves the first part 
of the theorem.\vspace{0.3cm}

An orientable Engel structure on $Q$ with twisting number $n$ 
exists if and only if there is a fiberwise $n/2$-fold covering
map $Q\lra\xi_1$. This is equivalent to the equality 
$n/2\cdot e(Q)=e(\xi_1)$ (see~Theorem~\ref{thm:coveringclassification}). This 
proves the second part of the theorem.\vspace{0.3cm}

To prove the third statement, observe that there is a simply transitive
action of $H^1(M;\Z)$ on $\eng^n_\xi(Q)$ defined by
\[
  \Phi\co
  H^1(M;\Z)\times\eng^n_\xi(Q)\lra\eng^n_\xi(Q),
  \;(\alpha,\phi_\mD)\lmt\alpha\cdot\phi_\mD
\]
where the product is defined as in \eqref{eq:action}, $\mD$ is an 
Engel structure on $Q$ and $\phi_\mD$ its development map 
(see~Theorem~\ref{thm:coveringclassification} and Theorem~\ref{thm:first}).
Here, we implicitly used the fact that Engel structures with characteristic line
field tangent to the fibers and fiberwise covering maps can be considered
as equivalent objects.
The oriented Engel structures sit in $\eng^n_\xi(Q)$ as a
subset. More precisely, the projection $\pi_\xi$ induces a map
\[
  \iota_\xi\co
  \mbox{\rm Cov}_{\frac{n}{2}}(Q,\xi_1)\lra\eng^n_\xi(Q),
  \,\phi\lmt\pi_\xi\circ\phi
\]
such that $\im(\iota_\xi)=\eng^n_{\xi;o}(Q)$. Suppose that $\phi_\mD$ 
factorizes through $\xi_1$ via $\phi_1$, then
\[
 (2\alpha)\cdot\phi_\mD
 =
 (2\alpha)\cdot(\pi_\xi\circ\phi_1)
 =
 \pi_\xi\circ (\alpha\cdot\phi_1).
\]
Therefore, the diagram\vspace{0.3cm}
\[
  \xymatrix@R=1.0cm@C=0.2cm{
  \alpha\ar@/^0.7cm/@{|->}[rrrrr]\ar@{|->}[d]&H^1(M;\Z)\ar[rrr]^{\cdot 2}\ar[d]_\cong 
  &&& H^1(M;\Z)\ar[d]^\cong &2\cdot\alpha\ar@{|->}[d]\\
  \alpha\cdot\phi_1\ar@/_0.7cm/@{|->}[rrrrr]&\mbox{\rm Cov}_{\frac{n}{2}}(Q,\xi_1)\ar[rrr]^{\iota_\xi} 
  &&& \eng^n_\xi(Q) & \pi_\xi\circ(\alpha\cdot\phi_1)=(2\alpha)\cdot\pi_\xi\circ\phi_1
  }\vspace{0.5cm}
\]
is commutative, which shows that $\Phi$ descends to a simply-transitive
$2\cdot H^1(M;\Z)$-action on $\im(\iota_\xi)=\eng^n_{\xi;o}(Q)$,
which finishes the proof.
\end{proof}
\begin{proof}[Proof of Corollary~\ref{cor:main}] Since the
Euler class $e(M\times\sone)$ of the trivial circle bundle vanishes,
the equation
$
  n\cdot e(M\times\sone)=2\cdot e(\xi)
$
is fulfilled for every number $n$ and contact structure $\xi$ with 
Euler class/first chern class a $2$-torsion class. So, 
Theorem~\ref{thm:main} states that
$\eng^n_\xi(M\times\sone)$ bijects onto $H^1(M;\Z)$ for all pairs
 $(n,\xi)\in \Z\times\Xi_2(M)$ which proves the first part of the
corollary.\vspace{0.3cm}

Suppose we are given an oriented Engel structure $\mD$ on $M\times\sone$.
Since $\mD$ is oriented it comes equipped with a trivialization. Since 
$M\times\sone$ is a trivial circle bundle, by choosing a section of this
circle bundle, the trivialization of $\mD$ induces a trivialization of the 
induced contact
structure $\xi$ on the base space. So, the isotopy classes of oriented Engel
structures can be thought of as a subset of $2\Z\times\Xi_0(M)\times H^1(M;\Z)$.
Conversely, every isotopy class of Engel structures inside the subset
$2\Z\times\Xi_0(M)\times H^1(M;\Z)$ fulfills the equation
\[
  n/2\cdot e(M\times\sone)=e(\xi)
\]
and is therefore represented by an oriented Engel structure due to 
Theorem~\ref{thm:main}.
\end{proof}

\end{document}